\documentclass[12pt]{article}

\usepackage{amssymb,amsmath,latexsym,graphics, supertabular}

\usepackage[enableskew]{youngtab}

\usepackage{amsbsy}

\usepackage{amscd}

\usepackage{verbatim}

\usepackage{amsfonts}

\usepackage{amsthm}

\usepackage{times}

\usepackage{mathrsfs}

\usepackage{verbatim}

\usepackage{graphicx,subfigure}

\usepackage[colorlinks]{hyperref}

\usepackage{xypic}

\usepackage{txfonts}

\usepackage[english]{babel}

\usepackage[T1]{fontenc}

\usepackage[utf8]{inputenc}

\usepackage{tikz}

\usetikzlibrary{backgrounds,fit,decorations.pathreplacing}
\usetikzlibrary{arrows,shapes,positioning}
\usetikzlibrary{calc,through,backgrounds,chains}

\usetikzlibrary{arrows,shapes,snakes,automata,backgrounds,petri}

\newenvironment{pf*}[1]{\proof[#1]}{\endproof}
\newtheorem{Theorem}[equation]{Theorem}
\newtheorem{Corollary}[equation]{Corollary}

\newtheorem{Proposition}[equation]{Proposition}

\theoremstyle{definition}
\newtheorem{Definition}[equation]{Definition}
\newtheorem{Example}[equation]{Example}

\theoremstyle{remark}

\newtheorem{Remark}[equation]{Remark}

\numberwithin{equation}{section}
\numberwithin{figure}{section}

\newcommand{\PP}{{\mathbb P}}
\newcommand{\C}{{\mathbb C}}

\newcommand{\R}{{\mathbb R}}

\newcommand{\mc}[1]{\mathcal{#1}}
\newcommand{\ms}[1]{\mathscr{#1}}
\newcommand{\mf}[1]{\mathfrak{#1}}

\newcommand{\beq}{\begin{equation}}
\newcommand{\eeq}{\end{equation}}

\newcommand{\mt}[1]{\text{#1}}

\def\red#1{{\color{red}{#1}}}

\def\ds#1{\displaystyle{#1}}

\begin{document}

\title{$\mt{SL}_2$-regular Subvarieties of Complete Quadrics}
\author{Mahir Bilen Can\\ Michael Joyce}
\maketitle

\begin{abstract}
We determine $\mt{SL}_n-$stable, $\mt{SL}_2-$regular subvarieties of the variety of complete quadrics.
We extend the results of Aky{\i}ld{\i}z and Carrell on Kostant-Macdonald identity by computing the 
Poincar{\'e} polynomials of these regular subvarieties.  
\end{abstract}

\section{Introduction}

The study of the variety $\ms{X}:=\ms{X}_n$ of $(n-2)$-dimensional complete quadrics, a completion of the variety of 
smooth quadric hypersurfaces in $\PP^{n-1}(\C)$, dates back to the nineteenth century, where it was used to answer 
fundamental questions in enumerative geometry.  Complete quadrics received renewed attention in the second 
half of the twentieth century for two primary reasons: (1) the toolkit of modern algebraic geometry made it possible to 
develop Schubert calculus rigorously, thereby addressing Hilbert's $15^{\text{th}}$ problem \cite{KL72, Laksov87}; 
and (2) the interpretation of $\ms{X}$ by De Concini and Procesi as an example of a wonderful embedding \cite{DP83}.

Identifying quadrics with the symmetric matrices defining them (up to scaling), the change of variables action of $\mt{SL}_n:=\mt{SL}_n(\C)$ 
on smooth quadrics corresponds to the action on symmetric matrices given by $g \cdot A = \theta(g) A g^{-1}$, for $g \in \mt{SL}_n$, 
$A$ a non-degenerate symmetric matrix, and $\theta$ the involution $\theta(g) = (g^\top)^{-1}$. Modulo the center of $\mt{SL}_n$, 
the variety of smooth quadric hypersurfaces can be identified as $\mt{SL}_n / \mt{SO}_n$ since $\mt{SO}_n \subset \mt{SL}_n$ 
is the stabilizer of the smooth quadric defined by the identity matrix.

Any semi-simple, simply connected complex algebraic group $G$ equipped with an involution $\sigma$ has a 
canonical wonderful embedding $X$. Letting $H$ denote the normalizer of $G^\sigma$, $X$ is a smooth projective 
$G-$variety containing an open $G-$orbit isomorphic to $G/H$ and whose boundary $X - (G/H)$ is a union of smooth 
$G-$stable divisors with smooth transversal intersections. Boundary divisors are canonically indexed by the elements 
of a certain subset $\varDelta$ of a root system associated to $(G,\sigma)$. Each $G-$orbit in $X$ corresponds to a subset 
$I\subseteq \varDelta$.  The Zariski closure of the orbit is smooth and is equal to the transverse intersection of the 
boundary divisors corresponding to the elements of $I$.

The  wonderful embedding of the symmetric pair $(\mt{SL}_n,\theta)$ above is $\ms{X}$, 
where $\sigma(A) = (A^{\top})^{-1}$. In this case, $\varDelta$ is the set of simple roots associated to $\mt{SL}_n$ 
relative to its maximal torus of diagonal matrices contained in the Borel subgroup $B\subseteq \mt{SL}_n$
of upper triangular matrices, and is canonically identified with the set $[n-1] = \{1, 2, \dots n-1\}$.

The study of cohomology theories of wonderful embeddings, initiated in \cite{DP83}, has been carried out through several 
different approaches. Poincar\'e polynomials have been computed in \cite{DS85, Strickland86, Renner03},
while the structure of the (equivariant) cohomology rings have been described in \cite{DGMP88, LP90, BDP90, Strickland06, BJ08}.

We study the cohomology of $\mt{SL}_n$-stable subvarieties of $\ms{X}$ that are $\mt{SL}_2$-regular.  An $\mt{SL}_2$-regular variety is one which admits an action of $\mt{SL}_2$ such that any one-dimensional unipotent subgroup of $\mt{SL}_2$ fixes a single point.
Aky{\i}ld{\i}z and Carrell developed a remarkable approach for studying the cohomology algebra $H^*(X;\C)$ of such varieties (\cite{ACLS83, AC87, AC89}).
Their method, when applied to flag varieties, has important representation theoretic consequences.

Let us briefly describe the contents of this paper.  Section \ref{S:preliminaries} sets some notation and recalls 
some basic facts about $\mt{SL}_2$-regular varieties, wonderful embeddings, and complete quadrics.  
In Section \ref{S:regular=special}, we precisely identify which $\mt{SL}_n$-stable subvarieties of $\ms{X}$ are $\mt{SL}_2$-regular.  
When combined with an earlier result of Strickland \cite{Strickland86}, our result takes an especially nice form: 
an $\mt{SL}_n$-stable subvariety of $\ms{X}$ is $\mt{SL}_2$-regular if and only if the dense orbit of the subvariety 
contains a fixed point of the maximal torus of $\mt{SL}_n$.

In Section \ref{S:KM}, we apply the machinery developed by Aky{\i}ld{\i}z and Carrell to compute the Poincar\'e polynomial
$$
P_{X}(t) := \sum_{i=0}^{2 \dim X}\dim_{\C} H^{i}( X;\C) t^i
$$
of a $\mt{SL}_2$-regular, $\mt{SL}_n$-stable subvariety $X$ of $\ms{X}$. 
If $I \subset [n-1]$ is the corresponding set of simple roots, then
\begin{equation*}
P_{X}(t) = \left( \frac{1-t^6}{1-t^4}\right)^{|I|} \prod_{k=1}^n \frac{1 - t^{2k}}{1 - t^2}.
\end{equation*}
This factorization of the Poincar\'e polynomial should be viewed as a generalization of the famous identity of 
Kostant and Macdonald (\cite{Kostant59, Macdonald72})
\begin{equation*}
\sum_{\pi \in S_n} q^{\ell(\pi)} = \prod_{k=1}^n \frac{1 - q^k}{1 - q},
\end{equation*}
recognizing the left-hand side as the Poincar\'e polynomial of the complete flag variety, evaluated at $q = t^2$.

\noindent \textbf{Acknowledgement.} The first author is partially supported by the Louisiana Board of Regents enhancement grant.

\section{Preliminaries}\label{S:preliminaries}
	
\subsection{Notation and Conventions}\label{S:notation}
All varieties are defined over $\C$ and all algebraic groups are complex algebraic groups.
Throughout, $n$ is a fixed integer, and $\ms{X} := \ms{X}_n$ denotes the $\mt{SL}_n$-variety of $(n-2)$-dimensional 
complete quadrics, which is reviewed in Section \ref{S:complete quadrics}.
The set $\{1, 2, \dots, m\}$ is denoted $[m]$ and if $I \subset [m]$, then its complement is denoted $I^\text{c}$.
If $I$ and $K$ are sets, then $I-K$ denotes the set complement $\{a\in I:\ a\notin K\}$.
The transpose of a matrix $A$ is denoted $A^\top$.

We denote by $B' \subset \mt{SL}_2$ the subgroup of upper triangular matrices, with its usual decomposition $B' = T' U'$ into 
a semidirect product of a maximal torus $T'$ consisting of the diagonal matrices and the unipotent radical $U'$ of $B'$. 
Let $\mathfrak{b}', \mathfrak{t}', \mathfrak{u}'$ denote their Lie algebras.

Finally, the symmetric group of permutations of $[n]$ is denoted by $S_n$, and for $w \in S_n$, 
$\ell(w)$ denotes $\ell(w)= |\{(i,j):\ 1\leq i <j \leq n,\ w(i) > w(j)\}|$.

\subsection{$\mt{SL}_2$-regular Varieties}\label{S:regular}

Let $X$ be a smooth projective variety over $\C$ on which an algebraic torus $T$ acts with finitely many fixed points.
Let $T'$ be a one-parameter subgroup with $X^{T'} = X^T$. For $p\in X^{T'}$ define the sets 
$C_p^+ = \{y \in X:\ \ds{\lim_{t \to 0} t \cdot y = p,}\ t \in T' \}$ and 
$C_p^- = \{y \in X:\ \ds{\lim_{t \to \infty} t \cdot y = p,}\ t \in T' \}$, 
called the {\em plus cell} and {\em minus cell} of $p$, respectively.

\begin{Theorem}[\cite{BB73}]\label{T:BB}
Let $X,T$ and $T'$ be as above. Then
\begin{enumerate}
\item $C_p^+$ and $C_p^-$ are locally closed subvarieties isomorphic to affine space;
\item if $T_p X$ is the tangent space of $X$ at $p$, then $C_p^+$ (resp., $C_p^-$) is $T'$-equivariantly 
isomorphic to the subspace $T_p^+ X$ (resp., $T_p^- X$) of $T_p X$ spanned by the positive (resp., negative) 
weight spaces of the action of $T'$ on $T_p X$.
\end{enumerate}
\end{Theorem}

As a consequence of Theorem \ref{T:BB}, there exists a filtration
$$
X^{T'}  = V_0 \subset V_1 \subset \cdots \subset V_n = X,\qquad n = \dim X,
$$
of closed subsets such that for each $i=1,\dots,n$, $V_i - V_{i-1}$ is the disjoint union of
the plus (resp.,  minus) cells in $X$ of (complex) dimension $i$. It follows that the odd-dimensional integral 
cohomology groups of $X$ vanish, the even-dimensional integral cohomology groups of $X$ are free, 
and the Poincar\'e polynomial $P_X(t) := \sum_{i=0}^{2n} \dim_{\C} H^{i}(X; \C) t^i$ of $X$ is given by
$$
P_X(t) = \sum_{p \in X^{T'}} t^{2 \dim C_p^+} = \sum_{p \in X^{T'}} t^{2 \dim C_p^-}.
$$
Because the odd-dimensional cohomology vanishes, we will prefer to study the $q$-Poincar\'e polynomial, $P_X(q) = P_X(t^2)$.

Now suppose that $X$ has an action of $\mt{SL}_2$.  The action of $U'$ gives rise to a vector field $V$ on $X$.  
Note that $p \in X$ is fixed by $U'$ if and only if $V(p) \in T_p X$ is zero.  The $\mt{SL}_2$-variety $X$ is said to 
be {\em $\mt{SL}_2$-regular} if there is a unique $U'$-fixed point on $X$.  An $\mt{SL}_2$-regular variety has only finitely many 
$T'$-fixed points \cite{AC89}.

A smooth projective $G$-variety $X$ is {\em $\mt{SL}_2$-regular} if there exists an injective homomorphism 
$\phi: \mt{SL}_2 \hookrightarrow G$ such that the induced action makes $X$ into an $\mt{SL}_2$-regular $\mt{SL}_2$-variety.  
Recall that the Jacobson-Morozov Theorem \cite[Section 5.3]{Carter85} implies that when $G$ is simply-connected 
(the only case we consider) such $\phi$ are determined by specifying $h \in \mathfrak{t'}$ and $e \in \mathfrak{u'}$ satisfying 
$[h,e] = 2e$.  As an abuse of notation, we will often identify $B', T', U'$ (resp., $\mathfrak{b}', \mathfrak{t}', \mathfrak{u}', h, e$) 
with their images under $\phi$ (resp., $d \phi$).

Let $p$ be the unique $U'$-fixed point of the $\mt{SL}_2$-regular variety $X$.  The minus cell $C_p^-$ is open in $X$ \cite{AC87}, 
and hence, the weights of $T'$ on $T_p X$ are all negative.  Let $x_1, \dots, x_n$ be a $T'$-equivariant basis for the cotangent 
space $T_p^* X$ of $X$.  Then the coordinate ring $\mathbb{C}[C_p^-] = \mathbb{C}[x_1, \dots, x_n]$ is a graded algebra with 
$\deg x_i > 0$.  Viewing the vector field $V$ associated to the $U'$ action as a derivation of $\mathbb{C}[x_1, \dots, x_n]$, 
$V(x_i)$ is homogeneous of degree $\deg x_i + 2$ and $V(x_1), V(x_2), \dots V(x_n)$ is a regular sequence in 
$\mathbb{C}[x_1, \dots, x_n]$ \cite{AC89}.

\begin{Theorem}\cite[Proposition 1.1]{AC87}\label{T:AC}
Let $Z$ be the zero scheme of the vector field $V$, supported at the point $p \in X$, and let 
$I(Z) = (V(x_1), \dots, V(x_n)) \subset \C[x_1, \dots, x_n]$ be the ideal of $Z$, graded as above. 
Then there exists a degree-doubling isomorphism of graded algebras $\C[C_p^-] / I(Z) \cong H^*(X; \C)$.

Consequently, the $q$-Poincar\'e polynomial of $X$ is given by
$$
P_X(q) = \prod_{i = 1}^n \frac{1 - q^{\deg x_i + 1}}{1 - q^{\deg x_i}}.
$$
\end{Theorem}

\subsection{Wonderful Embeddings}\label{S:wonderful embeddings}
		
We briefly review the theory of wonderful embeddings, referring the reader to \cite{DP83} and 
\cite{Pezzini10} for more details.

\begin{Definition}\label{D:wonderful embedding}
Let $X$ be a smooth, complete $G$-variety containing a dense open homogeneous subvariety $X_0$. 
Then $X$ is a {\em wonderful embedding} of $X_0$ if
\begin{enumerate}
\item $X - X_0$ is the union of finitely many $G$-stable smooth codimension one subvarieties $X_i$ for $i = 1, 2, \dots, r$;
\item for any $I \subset [r]$, the intersection $X^I := \cap_{i \notin I} X_i$ is smooth and transverse;
\item every irreducible $G$-stable subvariety has the form $X^I$ for some $I \subset [r]$.
\end{enumerate}
If a wonderful embedding of $X_0$ exists, it is unique up to $G$-equivariant isomorphism.
\end{Definition}

The $G$-orbits of $X$ are also parameterized by the sets $I \subset [r]$. We denote by $\ms{O}^I$ the unique dense 
$G$-orbit in $X^I$.  There is a fundamental decomposition
\begin{equation}\label{E:orbit union}
X^I = \bigsqcup_{K \subset I} \ms{O}^K.
\end{equation}
Note that $X$ contains a unique closed orbit $Z$, corresponding to $I = \emptyset$.

\begin{Remark}\label{R:sphericalroots}
Fix a Borel subgroup $B \subset G$ and let $B^-$ denote the opposite Borel subgroup of $B$.  
Fix a maximal torus $T \subset B$ and let $p \in Z$ be the unique $B^-$-fixed point.  
The {\em spherical roots} of $X$ are the $T$-weights of $T_p X / T_p Z$ and the set $[r]$ in 
Definition \ref{D:wonderful embedding} can be intrinsically identified with the set of spherical roots of $X$.
\end{Remark}

\subsection{Complete Quadrics}\label{S:complete quadrics}

There is a vast literature on the variety $\ms{X}$ of complete quadrics.  See \cite{Laksov87} for a survey, 
as well as \cite{DP83} and \cite{DGMP88} for recent work on the cohomology ring of $\ms{X}$.  We briefly recall the relevant definitions.

Let $X_0$ denote the open set of the projectivization of $\text{Sym}_n$, the space of symmetric $n$-by-$n$ matrices, 
consisting of matrices with non-zero determinant.  Elements of $X_0$ should be interpreted as (the defining equations of) 
smooth quadric hypersurfaces in $\PP^{n-1}$.  The group $\mt{SL}_n$ acts on $X_0$ by change of variables defining 
the quadric hypersurfaces, which translates to the action
\begin{equation}\label{E:action}
g \cdot A = (g^\top)^{-1} A g^{-1}
\end{equation}
on $\text{Sym}_n$.  $X_0$ is a homogeneous space under this $\mt{SL}_n$ action and the stabilizer of the quadric 
$x_1^2 + x_2^2 + \dots x_n^2 = 0$ (equivalently, the identity matrix) is $\mt{SO}_n$.

The classical definition of $\ms{X}$ (see \cite{Schubert, Semple48, Tyrrell56}) is as the closure of the image of the map
$$
[A] \mapsto ([A], [\Lambda^2(A)], \dots, [\Lambda^{n-1}(A)]) \in \prod_{i=1}^{n-1} \PP(\Lambda^i(\text{Sym}_n)).
$$
Renewed interest in the variety of complete quadrics can be attributed in large part to the following theorem, 
which gives two alternative descriptions of $\ms{X}$.

\begin{Theorem}\label{T:V DP}
\begin{enumerate}
\item \cite{Vainsencher82} $\ms{X}$ can be obtained by the following sequence of blow-ups:  in the naive compactification 
$\PP^{n-1}$ of $X_0$, first blow up the locus of rank $1$ quadrics; then blow up the strict transform of the rank $2$ quadrics; 
$\dots$; then blow up the strict transform of the rank $n-1$ quadrics.\label{T:vainsencher quadrics}
\item \cite{DP83} $\ms{X}$ is the wonderful embedding of $X_0$ and the spherical roots of $\ms{X}$ are the simple positive roots of 
the $\rm{A_n}$ root system. \label{T:DP quadrics}
\end{enumerate}
\end{Theorem}

From Theorem \ref{T:V DP}(\ref{T:vainsencher quadrics}), a point $\mathcal{P} \in \ms{X}$ is described by the data of a flag
\begin{equation}\label{E:flag}
\mathcal{F}: V_0 = 0 \subset V_1 \subset \dots \subset V_{s-1} \subset V_s =\C^n
\end{equation}
and a collection $\mathcal{Q} = (Q_1, \dots Q_s)$ of quadrics, where $Q_i$ is a quadric in $\PP(V_i)$ whose singular locus 
is $\PP(V_{i-1})$.  It is clear from Theorem \ref{T:V DP}(\ref{T:DP quadrics}) that $r$ of Definition \ref{D:wonderful embedding} is 
equal to $n-1$, and moreover, $i \in [n-1]$ corresponds to the simple root $\alpha_i := \varepsilon_i - \varepsilon_{i+1}$ in the $\rm{A_n}$ root sytem (see Remark \ref{R:sphericalroots}).

Additionally, for each $K \subset [n-1]$, the map $(\mathcal{F}, \mathcal{Q}) \mapsto \mathcal{F}$ is an 
$\mt{SL}_n$-equivariant projection
\begin{equation}\label{E:pi_K}
\pi_K : \ms{X}^K \rightarrow \mt{SL}_n / P_K,
\end{equation}
where $P_K$ is the standard parabolic subgroup associated with the roots corresponding to $K$. 
The fiber over $\mathcal{F} \in \mt{SL}_n / P_K$ is isomorphic to a product of varieties of complete quadrics of 
smaller dimension.

$\ms{O}^K$ consists of complete quadrics whose flag $\mathcal{F}$ satisfies   
$\{ \dim V_i : i = 1, 2, \dots, s-1 \} = K^{\text{c}}$. $\ms{X}^K$ is a wonderful embedding of $\ms{O}_K$, the variety of 
complete quadrics whose flag satisfies $\{ \dim V_i : i = 1, 2, \dots, s-1 \} \subset K^{\text{c}}$ \cite{DP83}.

In Figure \ref{fig:Cells} we depict the cell decomposition of $\ms{X}_3$, the variety of complete conics in $\PP^2$. 
Each colored disk represents a $B-$orbit and edges stand for the covering relations between closures of $B-$orbits. 
A cell is a union of all $B-$orbits of the same color. 
We include the label $I\subseteq \{1,2\}$, which indicates the $\mt{SL}_3-$orbit containing the given 
$B-$orbit. We use the label $T$ to indicate the presence of a fixed point under the maximal torus $T$ of $\mt{SL}_3$.

\begin{figure}[htp]
\centering 

\begin{tikzpicture}[scale=.56]

\node [shape=circle,draw,fill=blue!10!] at (0,0) (a) {$\emptyset,T$};

\node [shape=circle,draw,fill=purple] at (-6,5) (b1) {$\{\mathbf{1}\},T$};
\node [shape=circle,draw,fill=blue!30!] at (-2,5) (b2) {$\{\mathbf{2}\},T$};
\node [shape=circle,draw,fill=red!20!] at (2,5) (b3) {$\emptyset,T$};
\node [shape=circle,draw,fill=orange!50!] at (6,5) (b4) {$\emptyset,T$};

\node [shape=circle,draw,fill=blue!50!] at (-10,10) (c1) {$\{\mathbf{1}\},T$};
\node [shape=circle,draw,fill=red!20!]  at (-6,10) (c2) {$\{\mathbf{1}\}$};
\node [shape=circle,draw,fill=green!50!] at (-2,10) (c3) {$\{\mathbf{2}\},T$};
\node [shape=circle,draw,fill=orange!50!] at (2,10) (c4) {$\{\mathbf{2}\}$};
\node [shape=circle,draw,fill=red!50!] at (6,10) (c5) {$\emptyset,T$};
\node [shape=circle,draw,fill=yellow!50!] at (10,10) (c6) {$\emptyset,T$};

\node [shape=circle,draw,fill=green!50!] at (-10,15) (d1) {$\{\mathbf{1,2}\}$};
\node  [shape=circle,draw,fill=brown]  at (-6,15) (d2) {$\{\mathbf{1}\},T$};
\node [shape=circle,draw,fill=red!50!] at (-2,15) (d3) {$\{\mathbf{1}\}$};
\node  [shape=circle,draw,fill=green!15]  at (2,15) (d4) {$\{\mathbf{2}\},T$};
\node [shape=circle,draw,fill=yellow!50!] at (6,15) (d5) {$\{\mathbf{2}\}$};
\node  [shape=circle,draw,fill=yellow]  at (10,15) (d6) {$\mathbf{\emptyset},T$};

\node  [shape=circle,draw,fill=brown] at (-6,20) (e1) {$\{\mathbf{1,2}\}$};
\node  [shape=circle,draw,fill=green!15] at (-2,20) (e2) {$\{\mathbf{1,2}\}$};
\node  [shape=circle,draw,fill=yellow]  at (2,20) (e3) {$\{\mathbf{1}\}$};
\node  [shape=circle,draw,fill=yellow]  at (6,20) (e4) {$\{\mathbf{2}\}$};

\node [shape=circle,draw,fill=yellow] at (0,25) (f) {$\{\mathbf{1,2}\}$};

\filldraw[dashed, thick,fill=red] (a) to (b1);
\draw[dashed] (a) to (b2);
\draw[dashed] (a) to (b3);
\draw[dashed] (a) to (b4);

\draw[dashed] (b1) to (c1);
\draw[dashed] (b1) to (c2);

\draw[dashed] (b2) to (c3);
\draw[dashed] (b2) to (c4);

\draw[-, ultra thick] (b3) to (c2);
\draw[dashed] (b3) to (c3);
\draw[dashed] (b3) to (c5);
\draw[dashed] (b3) to (c6);

\draw[dashed] (b4) to (c1);
\draw[-,ultra thick] (b4) to (c4);
\draw[dashed] (b4) to (c5);
\draw[dashed] (b4) to (c6);

\draw[dashed] (c1) to (d1);
\draw[dashed] (c1) to (d2);
\draw[dashed] (c1) to (d3);

\draw[dashed] (c2) to (d2);
\draw[dashed] (c2) to (d3);

\draw[-,ultra thick] (c3) to (d1);
\draw[dashed] (c3) to (d4);
\draw[dashed] (c3) to (d5);

\draw[dashed] (c4) to (d1);
\draw[dashed] (c4) to (d4);
\draw[dashed] (c4) to (d5);

\draw[ultra thick] (c5) to (d3);
\draw[dashed] (c5) to (d4);
\draw[dashed] (c5) to (d6);

\draw[ultra thick] (c6) to (d5);
\draw[dashed] (c6) to (d2);
\draw[dashed] (c6) to (d6);

\draw[dashed] (d1) to (e1);
\draw[dashed] (d1) to (e2);

\draw[dashed] (d1) to (c2);

\draw[ultra thick] (d2) to (e1);
\draw[dashed] (d2) to (e3);

\draw[dashed] (d3) to (e2);
\draw[dashed] (d3) to (e3);

\draw[ultra thick] (d4) to (e2);
\draw[dashed] (d4) to (e4);

\draw[dashed] (d5) to (e1);
\draw[dashed] (d5) to (e4);

\draw[ultra thick] (d6) to (e3);
\draw[ultra thick] (d6) to (e4);

\draw[dashed] (e1) to (f);
\draw[dashed] (e2) to (f);
\draw[ultra thick] (e3) to (f);
\draw[ultra  thick] (e4) to (f);

\end{tikzpicture} 
\caption{Cell decomposition of the complete quadrics for $n=3$.}
\label{fig:Cells}
\end{figure}
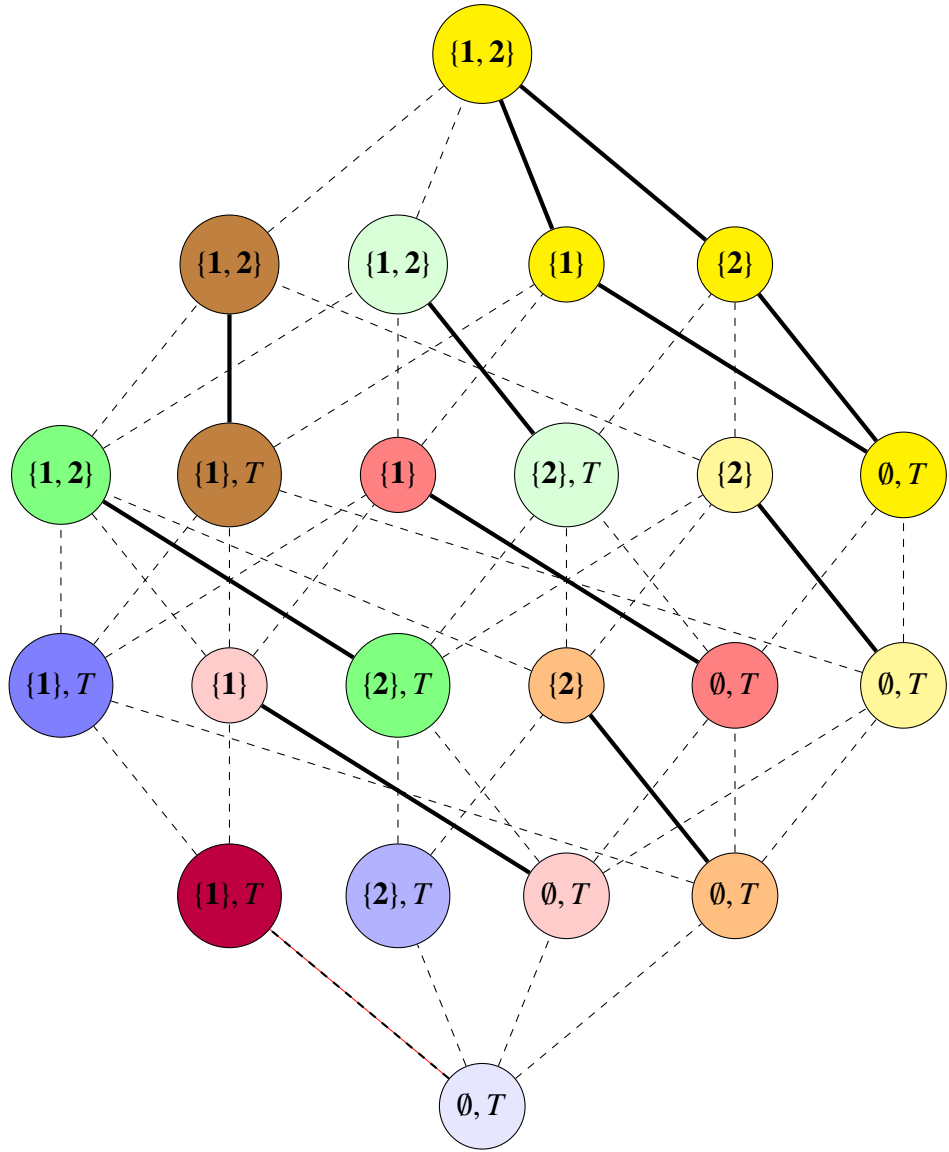

\subsection{Unipotent Fixed Flags}

We need the following elementary theorem on the fixed point loci of any partial flag variety $\mt{SL}_n / P$ 
under the action of a one dimensional unipotent subgroup $U' \hookrightarrow \mt{SL}_n$.  
Such loci are completely classified by Spaltenstein \cite{Spaltenstein76} and Shimomura \cite{Shimomura80}.

\begin{Theorem}\label{T:U' fixed flags}
Fix a one-dimensional unipotent subgroup $U' \hookrightarrow \mt{SL}_n$ with the Lie algebra $\mf{u}'= \mt{Lie}(U')$.
\begin{enumerate}
\item The fixed point locus $(\mt{SL}_n / P)^{U'}$ is non-empty.
\item If a non-zero element $e \in \mathfrak{u}'$ is regular, i.e. has a single Jordan block, 
then $(\mt{SL}_n / P)^{U'}$ consists of a single point.
\item If $(\mt{SL}_n / B)^{U'}$ is a single point, then any non-zero $e \in \mathfrak{u'}$ is regular.
\end{enumerate}
\end{Theorem}

\section{$\mt{SL}_2$-regular Subvarieties of Complete Quadrics}\label{S:regular=special}

\begin{Definition}\label{D:special subset}
A subset $I \subset [n-1]$ is \emph{special} if it does not contain any consecutive numbers.  
Equivalently, $I = \{ i_1 < i_2 < \dots < i_s \} \subset [n-1]$ is special if $i_{j+1} - i_{j} \geq 2$ for $j = 1, 2, \dots, s-1$.
\end{Definition}

\begin{Remark}\label{R:spec contain}
Given a special subset $I \subset [n-1]$, any subset $K \subset I$ is also special.
\end{Remark}

\begin{Theorem}\label{T:special = regular}
Let $I \subset [n-1]$.  The following are equivalent:
\begin{enumerate}
\item $I$ is special; \label{main thm 1}
\item $\mathscr{X}^I$ is $\mt{SL}_2$-regular; \label{main thm 2}
\item $\mathscr{O}^I$ contains a $T$-fixed point. \label{main thm 3}
\end{enumerate}
\end{Theorem}

\begin{Remark}
The equivalence $(\ref{main thm 1}) \Leftrightarrow (\ref{main thm 3})$ in Theorem \ref{T:special = regular} is due to Strickland \cite[Proposition 2.1]{Strickland86}.
\end{Remark}

\begin{proof}[Proof of $(\ref{main thm 1}) \Rightarrow (\ref{main thm 2})$]
Let $I$ be a special subset of $[n-1]$. Let
\begin{equation}\label{E:e}
e = \begin{pmatrix}0 & 1 & 0 & \dots & 0 \\
0 & 0 & 1 & \dots & 0 \\
\vdots & \vdots & \vdots & \ddots & \vdots \\
0 & 0 & 0 & \dots & 1 \\
0 & 0 & 0 & \dots & 0 \end{pmatrix}
\end{equation}
and
\begin{equation*}\label{E:h}
h = \begin{pmatrix}2n & 0 & 0 & \dots & 0 \\
0 & 2n - 2 & 0 & \dots & 0 \\
0 & 0 & 2n - 4 & \dots & 0 \\
\vdots & \vdots & \vdots & \ddots & \vdots \\
0 & 0 & 0 & \dots & 2 \end{pmatrix}.
\end{equation*}
A routine calculation shows that $[h, e] = 2e$, so let $\phi: \mt{SL}_2 \rightarrow \mt{SL}_n$ be the associated embedding.

Next we show that $\mathscr{X}^I$ is $\mt{SL}_2$-regular by proving that the unique $U'$-fixed point of $\mathscr{X}^I$ 
is the standard flag in $\C^n$, viewed as a point in $\mathscr{O}^{[n-1]} \cong \mt{SL}_n/B$.  
Since $\pi_K : \mathscr{O}^K \rightarrow \mt{SL}_n / P_K$ is $\mt{SL}_n$-equivariant (\ref{E:pi_K}), any $U'$-fixed point 
$\mc{P}=(\mathcal{F}, \mathcal{Q} = (Q_1, \dots, Q_s)) \in \mathscr{O}^K \subset \mathscr{X}^I$ maps to a $U'$-fixed partial flag 
$\mathcal{F} = \pi_K(\mc{P})$.  By Theorem \ref{T:U' fixed flags}, there is a unique $U'$-fixed partial flag $\mathcal{F}_K$ in each 
$\mt{SL}_n / P_K$.  Moreover, writing $K^{\text{c}} = \{k_1 < k_2 < \dots < k_t \}$, $\mathcal{F}_K$ is the flag whose $i$-th 
vector space is spanned by the first $k_i$ standard basis vectors.

For each $K \subset I$, we determine the $U'$-fixed locus of the fiber of $\pi_K$ over $\mathcal{F}_K$.  Since the flag
$$
\mathcal{F}_K := (0) = V_0 \subset V_1 \subset V_2 \subset \dots \subset V_{t-1} \subset V_t = \C^n
$$
is $U'$-fixed, the action of $u \in U'$ on a quadric $Q_i$ defined by the symmetric matrix $A_i$ in $V_i / V_{i-1}$ 
is given by restricting $u$ to a linear transformation on $V_i / V_{i-1}$. 
Because $K$ is special, $\dim V_i / V_{i-1} \leq 2$.  
Moreover, the matrix of $u$ with respect to the basis of standard basis vectors in 
$V_i- V_{i-1}$ is $ \begin{pmatrix} 1 & 1 \\ 0 & 1 \end{pmatrix}$ if $\dim (V_i / V_{i-1}) = 2$.

If
$$
A = \begin{pmatrix} a & b \\ b & c \end{pmatrix}
$$
defines such a quadric on a two-dimensional vector space, then
$$
u \cdot A = \begin{pmatrix} a & b-a \\ b-a & c-2b+a \end{pmatrix}.
$$
The only fixed quadric is degenerate and defined by
$A = \begin{pmatrix} 0 & 0 \\ 0 & 1 \end{pmatrix}.$

Therefore, if the point $(\mathcal{F}_K, \mathcal{Q})$ is $U'$-fixed, then each $V_i / V_{i-1}$ is one-dimensional.  
In other words, $K = \emptyset$ and $(\mathcal{F}_K, \mathcal{Q})$ is the standard flag in $\mt{SL}_n / B$.
\end{proof}

\begin{proof}[Proof of $(\ref{main thm 2}) \Rightarrow (\ref{main thm 1})$]
Assume that $I$ is not special.  Let $\phi : \mt{SL}_2 \rightarrow \mt{SL}_n$ be any homomorphism, 
giving rise to $B' = T' U' \subset \mt{SL}_n$.
To show that $\mathscr{X}^I$ is not regular, we must show that $U'$ does not have a unique fixed point.  
First, consider the action of $U'$ on $\mt{SL}_n / B$.  By Theorem \ref{T:U' fixed flags}, 
there are always $U'$-fixed flags and there is a unique $U'$-fixed flag if and only if any non-zero $e \in \mathfrak{u}'$ is regular.
Thus, we assume that the Jordan form of $e$ is 
$$\begin{pmatrix}
0 & 1 & 0 & \dots & 0 \\
0 & 0 & 1 & \dots & 0 \\
\vdots & \vdots & \vdots & \ddots & \vdots \\
0 & 0 & 0 & \dots & 1 \\
0 & 0 & 0 & \dots & 0
\end{pmatrix}.$$

Since $I$ is not special, there exists $a$, $1 \leq a < n - 1$, such that $a, a+1 \in I$.  Let $K = \{a, a+1\}$.  
Since $K \subset I$, $\mathscr{O}^K \subset \mathscr{X}^I$.  Moreover, $U'$-fixed points in $\mathscr{O}^K$ 
are in canonical bijection with the quadrics, defined on the three-dimensional vector space spanned by the $a$-th, 
$(a+1)$-st, and $(a+2)$-nd standard basis vectors, that are fixed by the restricted action of $U'$.  
Without loss of generality, we can assume that $U' \subset \mt{SL}_3$ and
$$ e = \begin{pmatrix} 0 & 1 & 0 \\ 0 & 0 & 1 \\ 0 & 0 & 0 \end{pmatrix} \in \mathfrak{u}'.$$

A standard Lie theory calculation using (\ref{E:action}) shows that a quadric defined by
$$A = \begin{pmatrix} a & b & c \\ b & d & e \\ c & e & f \end{pmatrix}$$
is fixed by $U'$ if and only if $e^{\top} A + A^{\top} e^{\top} = 0$ if and only if
$$A = \begin{pmatrix} 0 & 0 & c \\ 0 & -c & 0 \\ c & 0 & f \end{pmatrix}.$$

Thus, $\ms{O}^K$ contains a positive dimensional family of $U'-$fixed quadrics, and consequently 
$\mathscr{X}^I$ is not regular.

\end{proof}

\begin{Remark}
The proof that $(\ref{main thm 2}) \Leftrightarrow (\ref{main thm 3})$ in Theorem \ref{T:special = regular} is achieved by 
using the explicit combinatorics at hand to show that each of the two statements is equivalent to $(\ref{main thm 1})$.  
It is natural to wonder whether either direction of the implication holds in a more general setting.
\end{Remark}

\section{Poincar\'e polynomial of $\ms{X}^I$}\label{S:KM}

We apply Theorem \ref{T:AC} to compute the cohomology of $\mathscr{X}^I$ when $I$ is a special subset of $[n-1]$.

\begin{Proposition}\label{T:AC formulation}
If $I \subset [n-1] $ is special, then the $q$-Poincar\'e polynomial of $\mathscr{X}^I$ is equal to
\begin{align}\label{E:Poincare product}
P_{\mathscr{X}^I}(q) = \left( \frac{1-q^3}{1-q^2} \right) ^{|I|}
\prod_{k=1}^n \frac{1 - q^k}{1 - q}.
\end{align}
\end{Proposition}

\begin{proof}
Fix a regular $\mt{SL}_2$-action on $\mathscr{X}^I$ corresponding to
$$ e = \begin{pmatrix} 0 & 1 & 0 & \dots & 0 \\ 0 & 0 & 1 & \dots & 0 \\
\vdots & \vdots & \vdots & \ddots & \vdots \\ 0 & 0 & 0 & \dots & 1 \\
0 & 0 & 0 & \dots & 0 \end{pmatrix} \in \mathfrak{u} \text{ and }
h = \begin{pmatrix} n-1 & 0 & \dots & 0 \\ 0 & n-3 & \dots & 0 \\
\vdots & \vdots & \ddots & \vdots \\ 0 & 0 & \dots & -(n-1) \end{pmatrix} \in \mathfrak{t}
$$
so that $T'$ is included in $T$ via
$$
\begin{pmatrix} t & 0 \\ 0 & t^{-1} \end{pmatrix} \mapsto
\begin{pmatrix} t^{n-1} & 0 & \dots & 0 \\ 0 & t^{n-3} & \dots & 0 \\
\vdots & \vdots & \ddots & \vdots \\ 0 & 0 & \dots & t^{-(n-1)} \end{pmatrix}.
$$
From Theorem \ref{T:AC} and the discussion preceding it, to compute the Poincar\'e polynomial of $\mathscr{X}^I$, 
we must understand the $T'$-weight decomposition of the tangent space $T_{\mathcal{P}}(\mathscr{X}^I)$ of the unique 
$U'$-fixed point $\mathcal{P}$, corresponding to the standard flag in the complete flag variety 
$\mt{SL}_n / B \subset \mathscr{X}^I$. 
Since $T'$ is a subtorus of the maximal torus $T$ of $\mt{SL}_n$, consider the $T$-equivariant decomposition
$$
T_{\mathcal{P}}(\mathscr{X}^I) = T_{\mathcal{P}}(\mt{SL}_n / B) \oplus N_{\mathcal{P}}(\mt{SL}_n / B, \mathscr{X}^I).
$$
Here, $N_{\mathcal{P}}(\mt{SL}_n / B, \ms{X}^I)$ denotes the fiber of the normal bundle of $\mt{SL}_n / B$ in $\ms{X}^I$ 
at the point $\mathcal{P}$.

As a $T$-module, $T_{\mathcal{P}}(\mt{SL}_n / B) \cong \mathfrak{u}^{-} = {\ds \oplus_{\alpha > 0}} \mathfrak{u}_{-\alpha}$.  
Since $T'$ has weight $2$ acting on any simple positive root space, the weight of $T'$ on $\mathfrak{u}_{-\alpha}$ is 
$-2 \text{ht}(\alpha)$, where $\text{ht}(\alpha)$ is the height of $\alpha$ (c.f. \cite{Bourbaki}).  
The height of $\alpha = \varepsilon_i - \varepsilon_j$, $i< j$ is $j-i$.

Since $\mathscr{X}^I$ is a wonderful embedding of $\mathscr{O}^I$, $\mt{SL}_n / B$ is a transverse intersection of the 
$T$-stable subvarieties $\mathscr{X}^K$ where $K \subset I$ has cardinality $|I|-1$.  Thus, as a $T$-module,
$$
N_{\mathcal{P}}(\mt{SL}_n / B, \mathscr{X}^I) \cong \bigoplus_{j \in I} T_{\mathcal{P}}(\mathscr{X}^I) / T_{\mathcal{P}}(\mathscr{X}^{I - \{j\}}).$$
Since the $T$-weight of $T_{\mathcal{P}}(\mathscr{X}^I) / T_{\mathcal{P}}(\mathscr{X}^{I - \{j\}})$ is $-2 \alpha_j$ \cite{DP83}, its 
$T'$-weight is $-4$.

Combining these calculations with Theorem \ref{T:AC} gives (\ref{E:Poincare product}), using the elementary identity
$$
\prod_{\alpha > 0} \frac{ 1- q^{ht(\alpha)+1} }{1-q^{ht(\alpha)}} = \prod_{1 \leq i < j \leq n} \frac{1 - q^{j-i+1}}{1-q^{j-i}} = \prod_{k=1}^n \frac{1-q^{k}}{1-q}.
$$
\end{proof}

We interpret Theorem \ref{T:AC formulation} as a generalization of the classical Kostant-Macdonald identity 
(\cite{Kostant59, Macdonald72}) for the complete flag variety:
\begin{align}\label{A:kostantmacdonald}
\sum_{\pi \in S_n} q^{\ell(\pi)} = [n]_{q}! := \prod_{k=1}^n \frac{1 - q^k}{1-q}.
\end{align}
Aky{\i}ld{\i}z and Carrell recovered (\ref{A:kostantmacdonald}) as a corollary of Theorem \ref{T:AC} applied to the variety 
$X = \mt{SL}_n / B$.

In order to derive a similar ``sum = product'' identity in the case of the varieties $\mathscr{X}^I$, $I$ special, 
we compute $P_{\mathscr{X}^I}(q)$ by describing a decomposition into cells and applying Theorem \ref{T:BB}.  
To do so, we make use of a result of De Concini and Springer \cite{DS85} to reduce the calculation to that of a cell 
decomposition for $\mathscr{X}$, which was first computed by Strickland in \cite{Strickland86}.

Let $K \subset [n-1]$ be special and let $W = S_n$ be the symmetric group on $[n]$.
Let $W_K$ be the parabolic subgroup of $W$ generated by transpositions $(i, i+1)$ for $i \in K$, let $W^K$ be the 
set of minimal coset representatives of $W / W_K$, and let $w_{0,K} = \ds{\prod_{i \in K} (i, i+1)}$ denote the longest element of 
$W_K$. The $T$-fixed points in $\mathscr{O}^K$ are indexed by $W^K$ \cite[Proposition 2.3]{Strickland86}.

$W$ acts on the free abelian group generated by $\{ \varepsilon_1, \varepsilon_2, \dots, \varepsilon_n \}$ by 
$w \cdot \varepsilon_i = \varepsilon_{w(i)}$ and the simple roots $\alpha_i := \varepsilon_i - \varepsilon_{i+1}$ lie in this group.  
We write $v = \sum_{i=1}^n c_i \varepsilon_i > 0$ (resp., $< 0$) if the first non-zero coefficient $c_i$ 
which appears in the decomposition is positive (resp., negative).  Interpreting the $\varepsilon_i$ as characters of 
$\mt{SL}_n$, $v > 0$ is equivalent to the corresponding character being positive along a suitable one-dimensional torus 
$T' \subset \mt{SL}_n$. Define the set
\begin{equation*}\label{E:R set}
R_K(w) := \{ i \in K^{\text{c}} : w(\alpha_i + w_{0,K}(\alpha_i)) < 0 \}.
\end{equation*}

\begin{Proposition}[\cite{Strickland86}, Theorem 2.7 and Proposition 2.6]\label{P:Strickland}
Let $p \in \ms{O}^K$ be a $T$-fixed point of $\ms{X}$ corresponding to $w \in W^K$.  Let $C_p^+$ 
denote the plus cell of $p$ in $\ms{X}$ associated to the action of $T'$.  Then
$$
\dim C_p^+ = \ell(w) + |K| + |R_K(w)|.
$$
\end{Proposition}

\begin{Proposition}[\cite{DS85}, Lemma 4.1]\label{P:DS}
Retaining the notation of Propsition \ref{P:Strickland},
an orbit $\ms{O}^K$ intersects $C_p^+$ if and only if $K \subset I \subset R_K(w) \cup K$.
If $K \subset I$, then ${\ms{X}^I} \cap C_p^+$ is the plus cell of $\mathscr{X}^I$ containing $p$
and has dimension $\dim C_p^+ - |(I^{\text{c}} \cap R_K(w)|$.

\end{Proposition}

\begin{Theorem}\label{T:special Poincare}
Let $I \subset [n-1]$ be a special subset.  Then
$$
P_{\mathscr{X}^I}(q) = \sum_{K \subset I} \sum_{w \in W^K}  q^{\ell(w) + |K| + s_{K,I}(w)},
$$
where $s_{K,I}(w) = | \{i \in I - K : w(\alpha_i + w_{0,K}(\alpha_i)) < 0 \} |.$
\end{Theorem}

\begin{Corollary}\label{C:KM identity}
Let $I$ be a special subset of $[n-1]$. Then 
\begin{align}\label{E:KM}
\sum_{K \subset I} \sum_{ w\in W^K}  q^{\ell(w) + |K| + s_{K,I}(w)} =
\left( \frac{1-q^3}{1-q^2} \right)^{|I|} \prod_{k=1}^n \frac{1-q^{k}}{1-q}.
\end{align}
\end{Corollary}

\begin{Example}\label{E:KM n=3}
We illustrate Corollary \ref{C:KM identity} in the case $n = 3$.
If $I = \emptyset$, then we recover the classical Kostant-Macdonald identity for $\mt{SL}_3 / B$ (c.f. \cite{AC89}):
$$
1 + 2q + 2q^2 + q^3 = \frac{(1-q^2)(1-q^3)}{(1-q)^2} = (1+q)(1+q+q^2).
$$
If $I = \{1\}$, we obtain a new identity:
$$
(1 + q^2)(1 + q + q^2) + q(1 + q + q^2) = \frac{(1-q^3)}{(1-q^2)} \cdot \frac{(1-q^2)(1-q^3)}{(1-q)^2} = (1 + q + q^2)^2.
$$
The decomposition of the left-hand side reflects the sums over individual subsets $K \subset I$.
The identity for $I = \{2\}$ yields the same identity as $I = \{1\}$.
\end{Example}

\begin{Remark}
If $I$ is any special subset of $[n-1]$ of cardinality $l$ and $K \subset I$ has cardinality $k$, then one can show directly that
$$
\sum_{w \in W^K} q^{\ell(w) + |K| + s_{K,I}(w)} =
\left(\frac{q}{1+q^2}\right)^k \left(\frac{1+q^2}{1+q}\right)^{l} \prod_{i=1}^n \left(1+q+\dots+q^{i-1}\right)
$$
by verifying $s_{K,I}(w) = | \{ i \in I \setminus K : \ell(w s_i) < \ell(w) \} |$ (c.f. \cite[proof of Proposition 2.6]{Strickland86}).
Then (\ref{E:KM}) is obtained by summing over all $K \subset I$ and applying the Binomial Theorem.
\end{Remark}

\bibliography{Regular}

\begin{thebibliography}{10}

\bibitem{AC87}
E.~Aky{\i}ld{\i}z and J.B. Carrell.
\newblock Cohomology of projective varieties with regular {${\rm SL}_2$}
  actions.
\newblock {\em Manuscripta Math.}, 58(4):473--486, 1987.

\bibitem{AC89}
E.~Aky{\i}ld{\i}z and J.B. Carrell.
\newblock A generalization of the {K}ostant-{M}acdonald identity.
\newblock {\em Proc. Nat. Acad. Sci. U.S.A.}, 86(11):3934--3937, 1989.

\bibitem{ACLS83}
E.~Aky{\i}ld{\i}z, J.B. Carrell, D.I. Lieberman, and A.J. Sommese.
\newblock On the graded rings associated to holomorphic vector fields with
  exactly one zero.
\newblock In {\em Singularities, {P}art 1 ({A}rcata, {C}alif., 1981)},
  volume~40 of {\em Proc. Sympos. Pure Math.}, pages 55--56. Amer. Math. Soc.,
  1983.

\bibitem{BB73}
A.~Bia{\l}ynicki-Birula.
\newblock Some theorems on actions of algebraic groups.
\newblock {\em Ann. of Math. (2)}, 98:480--497, 1973.

\bibitem{BDP90}
E.~Bifet, C.~De~Concini, and C.~Procesi.
\newblock Cohomology of regular embeddings.
\newblock {\em Adv. Math.}, 82(1):1--34, 1990.

\bibitem{Bourbaki}
N.~Bourbaki.
\newblock {\em Lie groups and {L}ie algebras, {C}hapters 4--6}.
\newblock Elements of Mathematics (Berlin). Springer-Verlag, 2002.
\newblock Translated from the 1968 French original by Andrew Pressley.

\bibitem{BJ08}
M.~Brion and R.~Joshua.
\newblock Equivariant {C}how ring and {C}hern classes of wonderful symmetric
  varieites of minimal rank.
\newblock {\em Transform. Groups}, 13(3-4):471--493, 2008.

\bibitem{Carter85}
Roger~W. Carter.
\newblock {\em Finite groups of {L}ie type}.
\newblock Wiley Classics Library. John Wiley \& Sons Ltd., Chichester, 1993.
\newblock Conjugacy classes and complex characters, Reprint of the 1985
  original, A Wiley-Interscience Publication.

\bibitem{DGMP88}
C.~De~Concini, M.~Goresky, MacPherson R., and C.~Procesi.
\newblock On the geometry of quadrics and their degenerations.
\newblock {\em Comment. Math. Helv.}, 63(3):337--413, 1988.

\bibitem{DP83}
C.~De~Concini and C.~Procesi.
\newblock Complete symmetric varieties.
\newblock In {\em Invariant theory ({M}ontecatini, 1982)}, volume 996 of {\em
  Lecture Notes in Math.}, pages 1--44, Berlin, 1983. Springer.

\bibitem{DS85}
C.~De~Concini and T.A. Springer.
\newblock Betti numbers of complete symmetric varieties.
\newblock In {\em Geometry Today ({R}ome 1984)}, volume~60 of {\em Progr.
  Math.}, pages 87--107. Birkh{\" a}user Boston, 1985.

\bibitem{KL72}
S.L. Kleiman and D.~Laksov.
\newblock Schubert calculus.
\newblock {\em Amer. Math. Monthly}, 79:1061--1082, 1972.

\bibitem{Kostant59}
B.~Kostant.
\newblock The principal three-dimensional subgroup and the {B}etti numbers of a
  complex simple {L}ie group.
\newblock {\em American Journal of Mathematics}, 81:973--1032, 1959.

\bibitem{Laksov87}
D.~Laksov.
\newblock Completed quadrics and linear maps.
\newblock In {\em Algebraic geometry, {B}owdoin, 1985 ({B}runswick, {M}aine,
  1985)}, volume~46 of {\em Proc. Sympos. Pure Math.}, pages 371--387. Amer.
  Math. Soc., Providence, RI, 1987.

\bibitem{LP90}
P.~Littelmann and C.~Procesi.
\newblock Equivariant cohomology of wonderful compactifications.
\newblock In {\em Operator algebras, unitary representations, enveloping
  algebras, and invariant theory ({P}aris, 1989)}, volume~92 of {\em Progr.
  Math.}, pages 219--262. Birkh{\" a}user Boston, 1990.

\bibitem{Macdonald72}
I.G. Macdonald.
\newblock The {P}oincar\'e series of a {C}oxeter group.
\newblock {\em Math. Ann.}, 199:161--174, 1972.

\bibitem{Pezzini10}
G.~Pezzini.
\newblock Lectures on spherical and wonderful varieites.
\newblock In {\em Actions hamiltoniennes: invariants et classification}, volume
  1, no. 1 of {\em Les cours du CIRM}, pages 33--53, 2010.

\bibitem{Renner03}
L.~Renner.
\newblock An explicit cell decomposition of the wonderful compactification of a
  semisimple algebraic group.
\newblock {\em Canad. Math. Bull.}, 46(1):140--148, 2003.

\bibitem{Schubert}
H.~Schubert.
\newblock {\em Kalk{\" u}l der abz{\" a}hlenden Geometrie}.
\newblock Springer-Verlag, 1979.
\newblock Reprint of the 1879 original with an introduction by Steven L.
  Kleiman.

\bibitem{Semple48}
J.G. Semple.
\newblock On complete quadrics.
\newblock {\em J. London Math. Soc.}, 23:258--267, 1948.

\bibitem{Shimomura80}
N.~Shimomura.
\newblock A theorem on the fixed point set of a unipotent transformation on the
  flag manifold.
\newblock {\em J. of the Math. Soc. of Japan}, 32(1):55--64, 1980.

\bibitem{Spaltenstein76}
N.~Spaltenstein.
\newblock The fixed point set of a unipotent transformation on the flag
  manifold.
\newblock {\em Nederl. Akad. Wetensch. Proc. Ser. A}, 38(5):452--456, 1976.

\bibitem{Strickland86}
E.~Strickland.
\newblock Schubert-type cells for complete quadrics.
\newblock {\em Adv. Math.}, 62(3):238--248, 1986.

\bibitem{Strickland06}
E.~Strickland.
\newblock Equivariant cohomology of the wonderful group compactification.
\newblock {\em J. Algebra}, 306(2):610--621, 2006.

\bibitem{Tyrrell56}
J.A. Tyrrell.
\newblock Complete quadrics and collineations in {$S_n$}.
\newblock {\em Mathematika}, 3:69--79, 1956.

\bibitem{Vainsencher82}
I.~Vainsencher.
\newblock {S}chubert calculus for complete quadrics.
\newblock In {\em Enumerative geometry and classical algebraic geometry
  ({N}ice, 1981)}, volume~24 of {\em Progr. Math.}, pages 199--235. Birkh{\"
  a}user Boston, 1982.

\end{thebibliography}
\bibliographystyle{plain}

\end{document}